\documentclass[reqno, 10pt]{amsart}

\usepackage{enumitem} 

\usepackage[english]{babel}
\usepackage{amsmath}
\usepackage{amssymb}
\usepackage{amsthm}
\usepackage[numbers,sort&compress]{natbib}
\usepackage{amsthm}
\usepackage[pdfborder={0 0 0}, urlcolor=blue]{hyperref}
\usepackage{esvect}
\usepackage{color}
\usepackage{stmaryrd}

\usepackage{natbib}

\usepackage[textsize=footnotesize,color=yellow!70, bordercolor=white]{todonotes}

\usepackage
[a4paper, margin=1.5in]{geometry}

\usepackage[utf8]{inputenc}

\usepackage{pdflscape} 
\usepackage{afterpage}
\usepackage{etoolbox}

\hypersetup{colorlinks=true}

\newcommand{\comm}[1]{}

\DeclareRobustCommand{\SkipTocEntry}[5]{}

\numberwithin{subsection}{section}

\newcommand{\thistheoremname}{}
\newtheorem*{genericthm}{\thistheoremname}

\newcommand{\ZFC}{\mathsf{ZFC}}
\newcommand{\ZF}{\mathsf{ZF} }

\newcommand{\CH}{\mathsf{CH} }

\newcommand{\DC}{\mathsf{DC}}

\newcommand{\Col}{\ensuremath{\operatorname{Col}} }

\newcommand{\BS}{{}^\omega\omega}

\newcommand{\SIGMA}{\boldsymbol\Sigma}

\newcommand{\CC}{\mathbb{C}} 
\newcommand{\PP}{\mathbb{P}} 
\newcommand{\QQ}{\mathbb{Q}} 
\newcommand{\RR}{\mathbb{R}} 
\newcommand{\CF}{\mathbb{F}}

\newcommand{\Add}{\mathrm{Add}}

\newtheorem{theorem}{Theorem}[section]
\newtheorem*{theorem*}{Theorem} 
\newtheorem{lemma}[theorem]{Lemma}

\newtheorem{question}[theorem]{Question}

\theoremstyle{definition}

\newcounter{cl}
\newtheorem{claim}[cl]{Claim}
\newtheorem*{claim*}{Claim}
\newtheorem*{subclaim*}{Subclaim}
\newtheorem{definition}[theorem]{Definition}
\newtheorem{fact}[theorem]{Fact}

\theoremstyle{remark}

\newtheorem{case}{Case}
\newtheorem*{case*}{Case}

\newtheorem*{subcase*}{Subcase}

\newenvironment{enumerate-(a)}{\begin{enumerate}[label={\upshape (\alph*)}, leftmargin=2pc]}{\end{enumerate}}
 
\newenvironment{enumerate-(a)-r}{\begin{enumerate}[label={\upshape (\alph*)}, leftmargin=2pc,resume]}{\end{enumerate}}

\newenvironment{enumerate-(A)}{\begin{enumerate}[label={\upshape (\Alph*)}, leftmargin=2pc]}{\end{enumerate}}

\newenvironment{enumerate-(A)-r}{\begin{enumerate}[label={\upshape (\Alph*)}, leftmargin=2pc,resume]}{\end{enumerate}}

\newenvironment{enumerate-(i)}{\begin{enumerate}[label={\upshape (\roman*)}, leftmargin=2pc]}{\end{enumerate}}

\newenvironment{enumerate-(i)-r}{\begin{enumerate}[label={\upshape (\roman*)}, leftmargin=2pc,resume]}{\end{enumerate}}

\newenvironment{enumerate-(I)}{\begin{enumerate}[label={\upshape (\Roman*)}, leftmargin=2pc]}{\end{enumerate}}

\newenvironment{enumerate-(I)-r}{\begin{enumerate}[label={\upshape (\Roman*)}, leftmargin=2pc,resume]}{\end{enumerate}}

\newenvironment{enumerate-(1)}{\begin{enumerate}[label={\upshape (\arabic*)}, leftmargin=2pc]}{\end{enumerate}}

\newenvironment{enumerate-(1)-r}{\begin{enumerate}[label={\upshape (\arabic*)}, leftmargin=2pc,resume]}{\end{enumerate}}

\title{Uniformization and Internal Absoluteness}

\author{Sandra M\"uller} \address[Sandra M\"uller]{Sandra M\"uller, Institut f\"ur Diskrete Mathematik und Geometrie, TU Wien, Wiedner Hauptstra{\ss}e 8-10/104, 1040 Wien, Austria.}
\email{sandra.mueller@tuwien.ac.at}

\thanks{We would like to thank Joan Bagaria for sending us a preliminary version of \cite{Ba} and Menachem Magidor for a discussion about \cite[Theorem 2.1]{FeWo92}. 
We are grateful to the referee for detailed comments for improving the paper.} 

\thanks{The first-listed author gratefully acknowledges funding from L'OR\'{E}AL Austria, in collaboration with the Austrian UNESCO Commission and 
the Austrian Academy of Sciences - Fellowship \emph{Determinacy and Large Cardinals} and the support of the FWF Elise Richter grant number V844. 
This project has received funding from the European Union’s Horizon 2020 research and innovation programme under the Marie Sk\l odowska-Curie grant agreement No 794020 (Project IMIC) 
of the second-listed author. 
He also greatefully acknowledges partial support from FWF grant number I4039. 
This research was funded in whole or in part by EPSRC grant number EP/V009001/1 of the second-listed author.  For the purpose of open access, the authors have applied a ‘Creative Commons Attribution' (CC BY) public copyright licence to any Author Accepted Manuscript (AAM) version arising from this submission.
}

\author{Philipp Schlicht} \address[Philipp Schlicht]{School of Mathematics, 
University of Bristol, 
Fry Building.
Woodland Road, 
Bristol, BS8~1UG, UK \\ 
and 
Universit\"at Bonn, Mathematisches Institut, Endenicher Allee 60, 53115 Bonn, Germany}
\email{philipp.schlicht@bristol.ac.uk} 


\date{\today}

\begin{document}

\begin{abstract} 






Measurability with respect to ideals is 
tightly connected with absoluteness principles for certain forcing notions. We study a uniformization principle that postulates the existence of a uniformizing function on a large set, relative to a given ideal. We prove that for all $\sigma$-ideals $I$ such that the ideal forcing $\PP_I$ 
of Borel sets modulo $I$ is proper, this uniformization principle is equivalent to an absoluteness principle for projective formulas with respect to $\PP_I$ that we call \emph{internal absoluteness}. In addition, we show that it is equivalent to measurability with respect to $I$ together with $1$-step absoluteness for the poset $\PP_I$. These equivalences are new even for Cohen and random forcing and they are, to the best of our knowledge, the first precise equivalences between regularity and absoluteness beyond the second level of the projective hierarchy.






\end{abstract} 

\maketitle 


\section{Introduction}

The study of the connections between generic absoluteness and regularity properties has been a central theme in set theory in the past decades. For example, many results have appeared in work of Bagaria, Brendle, Halbeisen, Ikegami, Judah and others on regularity properties at the first levels of the projective hierarchy \cite{Ba91, BaJu95, Ba06, Br93, HaJu96, ikegami2009, ikegami2010, Ju93}, see \cite{Ba} for an overview of these results with background on generic absoluteness. Moreover, this is closely related to the work of Feng, Magidor and Woodin on universally Baire sets \cite{FeWo92}. For example, Bagaria showed that $\Sigma^1_3$ absoluteness for Cohen forcing is equivalent to the Baire property for all ${\bf\Delta}^1_2$ sets of reals  \cite{Ba91}.\footnote{See also \cite[Section 5]{Ba06} and \cite[Theorem 1.1]{ikegami2010}.} 
Analogously, $\Sigma^1_3$ absoluteness for Random forcing is equivalent to the statement that all ${\bf\Delta}^1_2$ sets of reals are Lebesgue measurable \cite{Ba91}.

The regularity property we study here is uniformization up to a small set, where ``small'' refers to being in some fixed ideal $I$. This has, for example, already been considered by Solovay in \cite[Theorem 1]{MR0265151} for the ideal of meager sets and the ideal of null sets. Moreover, Shelah showed in \cite{Sh84} that projective uniformization up to a meager set holds in his model for ``all sets of reals have the Baire property and $\CH$ fails'' (that is obtained as a forcing extension of a $\ZFC$ model), see also \cite{JR93}. 
Zapletal studied a more general notion for ideals $I$ in the context of idealized forcing \cite[Proposition 2.3.4]{MR2391923}. 
Uniformization up to Ramsey-small sets also plays a crucial role in Schrittesser's and T\"ornquist's celebrated result on infinite maximal almost disjoint families \cite{schrittesser2019ramsey}. 

In general, we define projective uniformization up to a small set as follows. 
It is the special case of the next definition where $\Gamma$ denotes the class of projective sets. 
We let $I$ always denote a $\sigma$-ideal on the Borel subsets of a Polish space $X$ that contains all singletons. By $\PP_I$ we denote the preorder \emph{mod $I$} on the $I$-positive Borel sets, i.e., those not in $I$. 
Moreover, let $p[R]=\{ x\in X\mid \exists y\ (x,y) \in R \}$ denote the projection of a relation $R\subseteq X\times X$. 


\begin{definition} \label{def:projective uniformization up to I}
For any $\sigma$-ideal $I$ on a Polish space $X$ and any class $\Gamma$ of subsets of $X$, \emph{$\Gamma$-uniformization up to $I$} denotes the statement: \smallskip \\ 
For any binary relation $R\in \Gamma$ on $X$ and for any Borel set $A \notin I$, there is a Borel subset $B\notin I$ of $A$ such that either: 
\begin{enumerate-(1)} 
\item 
$B$ is disjoint from $p[R]$, or otherwise 
\item 
$B$ is a subset of $p[R]$ and there is a Borel measurable function $f\colon B\rightarrow X$ whose graph is a subset of $R$. 
\end{enumerate-(1)} 
\end{definition} 

The main result of this paper is that projective uniformization up to $I$ is equivalent to an absoluteness principle for the forcing $\PP_I$ for all $\sigma$-ideals $I$ such that $\PP_I$ is proper. More precisely, we consider a principle of \emph{internal projective absoluteness} that postulates projective absoluteness between the universe and generic extensions of countable elementary submodels of some $H_\theta$. This is defined more formally in Definition \ref{def:IA} below. Variants of this notion for the class of all forcings were used in proofs in inner model theory, see \cite[\textsection 5]{St09}, and more recently for the class of all proper forcings in Neeman's and Norwood's triangular embedding theorem \cite[Theorem 22]{NeemanNorwood2019}. 
Chan and Magidor implicitly use instances of internal projective absoluteness established through tree representations \cite{chan2022relation}. 
We rediscovered this principle for an application to selectors for ideals \cite{MSSW}. 
Moreover, it turns out that it is equivalent to a variant of universal Baireness for formulas (see Theorem \ref{thm:IAEquivTreeable}).

We prove that in addition to internal absoluteness, also 1-step absoluteness together with $I$-measurability is equivalent to uniformization up to $I$.

\begin{definition}[Khomskii \cite{Kh12}]  
Let $I$ be a $\sigma$-ideal on a Polish space $X$ and let $A \subseteq X$. We say $A$ is \emph{$I$-measurable}\footnote{This is often called \emph{$I$-regular} in the literature.
} if for every Borel set $B \notin I$, there is some Borel set $C\notin I$ with $C\subseteq B$ and either $C\cap A=\emptyset$ or $C\subseteq A$. 
\end{definition}

To see that projective uniformization up to $I$ implies $I$-measurability for all projective sets, take a projective set $A$ and an arbitrary Borel set $B \notin I$ and apply uniformization up to $I$ to the characteristic function 
of $A$ restricted to $B$. 
For more details, see the proof of Theorem \ref{main level-by-level equivalence}.

We can now state the main result of this article.
 
 \begin{theorem} \label{characterization of internal absoluteness via 1-step absoluteness and regularity} 
The following statements are equivalent for all proper forcings $\PP_I$: 
\begin{enumerate-(1)}
\item 
\label{characterization of internal absoluteness via 1-step absoluteness and regularity 1} 
 Internal projective $\PP_I$-absoluteness holds. 
\item 
\label{characterization of internal absoluteness via 1-step absoluteness and regularity 2} 
Projective uniformization holds up to sets in $I$. 
\item 
\label{characterization of internal absoluteness via 1-step absoluteness and regularity 3} 
$1$-step $\PP_I$-absoluteness holds and all projective sets are $I$-measurable. 
\end{enumerate-(1)} 
\end{theorem} 

We will in fact prove a level-by-level version of this theorem, i.e., internal $\Sigma^1_n$ absoluteness is equivalent to uniformization of $\SIGMA^1_n$ relations up to $I$ for all $n\geq 1$. 

In the final part of this paper, we look at examples of this equivalence for Cohen and random forcing that lead to consistency strength results for internal projective absoluteness for these forcing notions. Moreover, we briefly discuss the strength of internal projective absoluteness for the class of all forcings.

\section{Uniformization}




For the class of projective sets, the uniformization principle in Definition \ref{def:projective uniformization up to I} coincides with the notion studied by Solovay \cite[Theorem 1]{MR0265151} and Shelah \cite{Sh84} for the ideal $I$ of meager sets. 
The latter states that any projective relation $R$ on $X$ with $p[R]=X$ can be uniformized by a Borel measurable function on a comeager set. 
More generally, we call a class $\Gamma$ of subsets of $X$ and its finite products \emph{sufficiently closed} if $\Gamma$ is closed under images and preimages of projections, complements, finite unions, products and contains all singletons. 

\begin{lemma} 
\label{uniformisation for sufficiently closed classes} 
Suppose that $\Gamma$ is sufficiently closed. 
Then $\Gamma$-uniformization up to $I$ is equivalent to its restriction to relations $R$ with $p[R]=X$. 
%
\end{lemma} 
\begin{proof} 
Suppose that $R$ is a relation in $\Gamma$. 
We can assume there exists some $y_0\in X$ such that $(x,y_0)\notin R$ for all $x\in X$, since we can otherwise replace $X$ with a larger space such as $X\times X$ and use that $\Gamma$ is closed under preimages of projections. 
Let $S=R\cup ((X \setminus p[R]) \times \{y_0\})$ and pick a Borel measurable uniformization $f\colon A\rightarrow X$ of $S$ up to $I$ for some Borel set $A\notin I$. 
Now split $A$ into the Borel sets $A_0=f^{-1}[\{y_0\}]$ and $A_1=f^{-1}[X\setminus \{y_0\}]$. 
Since $A\notin I$ and $I$ is an ideal, one of $A_0$, $A_1$ is not in $I$. 
This shows uniformization of $R$ up to $I$. 
\end{proof}

\begin{lemma} 
Projective uniformization up to meager sets coincides with Solovay's notion from \cite[Theorem 1]{MR0265151}. 
More generally, suppose that $\Gamma$ is sufficiently closed and $\PP_I$ is c.c.c. 
Then $\Gamma$-uniformization up to $I$ is equivalent to the statement: \smallskip \\ 
Any relation $R$ in $\Gamma$ with $p[R]=X$ can be uniformized by a Borel measurable function $f\colon A\rightarrow X$, where $A$ is a Borel set with $X\setminus A\in I$. 
%
\end{lemma} 
\begin{proof} 
Suppose that $R$ is a relation in $\Gamma$. 
Construct a sequence $\langle X_\alpha, A_\alpha, f_\alpha \mid \alpha<\lambda\rangle$ for some countable ordinal $\lambda$ such that $A_\alpha\notin I$ is a Borel set, $X_\alpha=X\setminus \bigcup_{\beta<\alpha} A_\beta$, the graph of $f_\alpha\colon A_\alpha \rightarrow X$ is a subset of $R$ for all $\alpha<\lambda$, and $A_\alpha\cap A_\beta=\emptyset$ for all $\alpha<\beta<\lambda$. 
In step $\alpha$, define $f_\alpha$ by applying uniformization up to $I$ to $R$ and $X_\alpha$. 
Since $\PP_I$ is c.c.c., we have $X_\lambda\in I$ for some countable ordinal $\lambda$. 
Then $f=\bigcup_{\alpha<\lambda} f_\alpha$ unformizes $R$ on a set with complement in $I$. 
The converse follows from Lemma \ref{uniformisation for sufficiently closed classes}. 
\end{proof}


Uniformization up to small sets also plays a crucial role in Schrittesser's and T\"ornquist's recent result on maximal almost disjoint families \cite[Theorem 1.1]{schrittesser2019ramsey}. 
For the ideal $I_R$ of Ramsey null sets, one can see that uniformization up to $I_R$ coincides with the conjunction of $I_R$-measurability and the principle $R$-Unif defined in \cite[Section 2.3]{schrittesser2019ramsey}. 
Their main result can be thus restated as follows: Assuming $\ZF+ \DC$ and uniformization up to $I_R$ for all sets of reals, there is no infinite maximal almost disjoint family. 


In general, the uniformization principle in Definition \ref{def:projective uniformization up to I} can be equivalently formulated as the conjunction of the conditions: 
\begin{enumerate-(1)} 
\item 
All sets of the form $p[A]$ for $A\in \Gamma$ are $I$-measurable. 
\item 
If $R\in \Gamma$ is a relation and $A\notin I$ is a Borel subset of $p[R]$, then there is a Borel subset $B\notin I$ of $A$ and a Borel measurable function $f\colon B\rightarrow X$ whose graph is contained in $R$. 
\end{enumerate-(1)}

\section{Internal Absoluteness}

In this section we define and discuss internal absoluteness, the absoluteness principle that characterizes uniformization up to a small set. We will consider absoluteness for projective formulas with real parameters. Similar notions have been defined in \cite{St09} and independently in \cite{MSSW}.



\begin{definition}[\emph{Internal absoluteness}] \label{def:IA}
Let $\CF$ be a class of forcing notions. We say \emph{internal projective $\CF$-absoluteness} holds if and only if the following holds for any sufficiently large regular cardinal $\theta$ and a club of countable models $M\prec H_\theta$: 
if $\PP\in M\cap \CF$, $\bar{M}$ is the transitive collapse of $M$ and $\bar{\PP}$ is the image of $\PP$ under the collapse, then for all $\bar{\PP}$-generic filters $g \in V$ over $\bar M$,
 \[ \bar M[g] \models \varphi(x) \Longleftrightarrow H_\theta \models \varphi(x), \] for every projective formula\footnote{A formula is projective if its quantifiers range over reals.} $\varphi(v_0)$ and every real $x$ in $\bar M[g]$.
\end{definition}

We analogously define internal $\Sigma^1_n$ absoluteness for $n\geq 1$. For a 
notion of forcing $\PP$, we say that \emph{internal projective $\PP$-absoluteness} holds if $\CF=\{\PP\}$ in the previous definition, and \emph{internal projective absoluteness} if $\CF$ equals the class of all forcings. 
We will frequently use that one can replace $H_\theta\models \varphi(x)$ by $V\models \varphi(x)$ in the definition of internal absoluteness, since $\varphi$ is projective. 


The rest of this section is devoted to general properties of the internal absoluteness principle. Before we argue that internal absoluteness is consistent from large cardinals, we want to remark that internal absoluteness for all forcings $\PP$ is related to the notion of being universally Baire. Before we give some details on this, we recall its definition. 

\begin{definition}[Feng, Magidor, Woodin \cite{FeWo92}] 
A subset $A$ of a topological space $Y$ is \emph{universally Baire} if $f^{-1}(A)$ has the property of Baire in any topological space $X$, where $f\colon X\rightarrow Y$ is continuous.\footnote{Two definitions appear in \cite{FeWo92}, one with respect to all topological spaces and one with respect to topological spaces with a regular open basis, but no proof is provided that the two definitions are equivalent. 
See \cite{ElekesPalfy2021} for a proof. 
As a historical note, Menachem Magidor attributed the definition of universally Baire to Schilling and Vaught \cite{MR704624} in a discussion in 2019. 
However, the first explicit definition in the literature is to our knowledge due to Feng, Magidor and Woodin.} 
\end{definition} 

The following equivalence provides a formulation of being universally Baire that has proven to be very useful in set theory, see also \cite[Definition 32.21]{Je03} or \cite[Definition 8.6]{Sch14}. To state the equivalence, we need another definition.

\begin{definition}
    Let $S$ and $T$ be trees on $\omega \times \kappa$ for some ordinal $\kappa$ and let $\lambda$ be an ordinal. We say  \emph{$(S,T)$ is $\lambda$-absolutely complementing} if \[ p[S] = \BS \setminus p[T] \] in every $\Col(\omega,\lambda)$-generic extension of $V$.
\end{definition}


\begin{lemma}\cite{FeWo92, ElekesPalfy2021}
A set of reals $A$ is universally Baire if and only if for every ordinal $\lambda$, there are $\lambda$-absolutely complementing trees $(S,T)$ with $p[S] = A$.
\end{lemma}

The above notions of internal absoluteness are closely connected to the following strengthening of being universally Baire.

\begin{definition} 
For a forcing notion $\PP$ we say a formula $\varphi(v_0,v_1)$ is \emph{treeable with respect to $\PP$} if and only if for every parameter $a \in V$ there are trees $S$ and $T$ such that in every $\PP$-generic extension $V[G]$ of $V$, \[ p[S] = \{ x \in \BS \mid \varphi(x,a)\} \text{ and } p[T] = \{ x \in \BS \mid \neg\varphi(x,a)\}.\] 
Moreover, we say $\varphi(v_0,v_1)$ is \emph{treeable} if it is treeable with respect to $\PP$ for all forcing notions $\PP$.
\end{definition} 

It is easy to see that if all projective formulas are treeable, then all projective sets are universally Baire. The converse is a well-known open question, see \cite[Question 1, Section 6]{FeWo92}.
\cite[Question 1, Section 6]{FeWo92} also asks whether projective absoluteness follows from the statement that all projective sets are universally Baire. The converse was asked in \cite[Introduction]{Wi17}. Internal absoluteness is a natural absoluteness principle that is equivalent to the statement that all projective formulas are treeable. See \cite[Lemma 5.1]{St09} for a proof of the following theorem due to Steel and Woodin. Their proof in fact also shows the ``local version'' of this theorem, i.e.,that internal projective $\PP$-absoluteness holds if and only if all projective formuals are treeable with respect to $\PP$.

\begin{theorem}\label{thm:IAEquivTreeable}
Internal projective absoluteness holds for $\PP$ if and only if all projective formulas are treeable with respect to $\PP$.
\end{theorem}

Many standard proofs for forcing generic absoluteness from large cardinals, as for example in 
\cite{Wi}, \cite{MR1856755} or \cite{NeemanNorwood2019} also show that internal absoluteness is consistent from large cardinals. The reason is that they prove that the following property holds which, as we argue below, in turn implies internal absoluteness.

\begin{definition} 
Let $W$ be an inner model and $\kappa$ a cardinal. The \emph{${<}\kappa$-generics property over $W$ in $V$} of a forcing $\PP$ is the following statement: 
For any $\PP$-generic extension $V[G]$ of $V$ and any real $x\in V[G]$, there exist  
\begin{enumerate-(1)} 
\item
a forcing $\QQ\in W$ with $|\QQ|^W<\kappa$ and 
\item 
a $\QQ$-generic filter $H\in V[G]$ over $W$ 
\end{enumerate-(1)} 
such that $x\in W[H]$. 
\end{definition} 

When $\kappa$ is clear from the context, we will also refer to the ${<}\kappa$-generics property as the \emph{small generics property}. 

\begin{lemma}\label{lem:small generics property}
Suppose that $\kappa$ is inaccessible and $G$ is $\Col(\omega,{<}\kappa)$-generic over $V$. 
If $\PP$ has the ${<}\kappa$-generics property over $V$ in $V[G]$, then every projective formula $\varphi(v_0,v_1)$ is treeable in $V[G]$ with respect to $\PP$.
\end{lemma} 
\begin{proof} 

Suppose that $\kappa$ is inaccessible and $G$ is $\Col(\omega,{<}\kappa)$-generic over $V$. We are going to use the following standard claim about $\Col(\omega,{<}\kappa)$-generic extensions of $V$.

\begin{claim}\label{cl:standard fact Levy collapse inaccessible}
Suppose that $\PP$ is a forcing in $V$ of size ${<}\kappa$ in $V$, and $H$ is $\PP$-generic over $V[G]$. 
Moreover, suppose that $\QQ$ is a forcing in $V$ of size ${<}\kappa$ in $V$ and $G_0$ is $\QQ$-generic over $V$. 
Then there is a $\Col(\omega,{<}\kappa)$-generic filter $H_0$ over $V[G_0]$ with $\RR^{V[G_0][H_0]}=\RR^{V[G][H]}$. 
\end{claim} 

Now suppose $\PP$ has the ${<}\kappa$-generics property over $V$ and let $\varphi(v_0,v_1)$ be a projective formula. Let $a \in V$ be a parameter. We will construct a tree $S$ such that \[ p[S] = \{ x \in \BS \mid \varphi(x,a)\} \] in every $\PP$-generic extension of $V[G]$. The tree $T$ for $\neg\varphi(x,a)$ can then be defined analogously. We obtain the tree $S$ as a union of trees $S_{\QQ,\sigma}$ for all forcings $\QQ\in V$ of size ${<}\kappa$ in $V$ and all $\QQ$-names $\sigma$ for reals. Here the tree $S_{\QQ,\sigma}$ is the canonical tree searching for a pair $(G_0,x)$, such that
\begin{enumerate-(1)}
    \item $x$ is a real,
    \item $G_0$ is a $\QQ$-generic filter over $V$,
    \item $\sigma^{G_0}=x$, and
    \item In  $V[G_0]$,  $1_{\Col(\omega,{<}\kappa)} \Vdash \varphi(\sigma^{\check{G}_0},\check{a})$. 
\end{enumerate-(1)}
See, for example, the proof of \cite[Lemma 5.1]{St09} for a formal definition of a similar search tree.

\begin{claim} 
$p[S] = \{ x \in \BS \mid \varphi(x,a)\}$ holds in all $\PP$-generic extensions of $V[G]$.
\end{claim} 
\begin{proof} 
Let $H$ be $\PP$-generic over $V[G]$. Suppose that $(G_0,x)$ is a branch through the tree $S$, say it is a branch through $S_{\QQ,\sigma}$. Then there is some $H_0$ as in Claim \ref{cl:standard fact Levy collapse inaccessible}. 
Since $(G_0,x)$ is a branch through $S_{\QQ,\sigma}$, we have $V[G_0][H_0]\models \varphi(x,a)$. 
As $\RR^{V[G_0][H_0]}=\RR^{V[G][H]}$, this implies $V[G][H]\models \varphi(x,a)$. 


Conversely, suppose that $V[G][H]\models \varphi(x,a)$. Using the ${<}\kappa$-generics property for $\PP$ over $V$ in $V[G]$, we find some $\QQ\in V$ of size ${<}\kappa$ in $V$ and some $\QQ$-generic filter $G_0\in V[G][H]$ over $V$ with $x\in V[G_0]$. Let $\sigma$ be a $\QQ$-name for a real such that $\sigma^{G_0} = x$. Then $(G_0,x)$ is a branch through $S_{\QQ,\sigma}$. 
\end{proof} 

We can now finish the proof of Lemma \ref{lem:small generics property} by analogously defining a tree $T$ such that \[ p[T] = \{ x \in \BS \mid \neg\varphi(x,a)\} \] holds in all $\PP$-generic extensions of $V[G]$.
\end{proof}

\section{Proof of the main theorem} 

This section is devoted to the proof of Theorem \ref{characterization of internal absoluteness via 1-step absoluteness and regularity}. 
We first introduce some terminology. 
Fix a $\sigma$-ideal $I$ on the Borel subsets of an uncountable Polish space $X$. 
We can assume that $X$ equals the Cantor space $2^\omega$, since all uncountable Polish spaces are Borel isomorphic \cite[Theorem 15.6]{kechris2012classical}. 
Elements of $2^\omega$ can be identified with sets of natural numbers and are called \emph{reals}.
A generic filter $G$ over $V$ for the forcing $\PP_I$ adds a real $x\in X$ such that for every Borel set $B\subseteq X$ coded in the ground model, $B\in G$ if and only if $x\in B$ by \cite[Proposition 2.1.2]{MR2391923}, and in particular $V[x]=V[G]$. 
We therefore call $x$ a \emph{$\PP_I$-generic real}. 

Suppose that $M$ is a $\omega$-model\footnote{We only consider wellfounded models.} of $\ZFC^-$ with $\PP_I\in M$ and 
$$\tau = \{ (\check{n},p) \mid p\in A_n)\}\in H_{\omega_1}^M$$ 
is a nice $\PP_I$-name for a real, where each $A_n$ is a countable antichain in $\PP_I$. 
If $g$ is a $\PP_I$-generic filter over $M$ that induces the $\PP_I$-generic real $x$, we write $\tau^x=\tau^g$. 
One can compute $\tau^x$ as follows. 
Define 
$$\tau^{(y)}= \{ n\in \omega \mid \exists p\ (\check{n},p)\in \tau \wedge    y\in p \}$$ 
for arbitrary reals $y$. 
By the definition of $x$ from $g$, $\tau^x=\tau^{(x)}$. 
Moreover, the map sending any real $y$ to $\tau^{(y)}$ is Borel measurable, since $n\in \tau^{(y)} \Longleftrightarrow y \in \bigcup A_n$ for all $n\in\omega$.

If $M$ is a countable $\omega$-model of $\ZFC^-$ and $\PP_I\in M$, we will write $A_{\PP_I,M}$ for the set of $\PP_I$-generic reals over $M$. 
Note that 
$$A_{\PP_I,M}=\bigcap \{\bigcup D\cap M\mid M \models ``D  \text{ is a dense subset of $\PP_I$''}\}$$ is Borel. 
From now on, assume that $\PP_I$ is proper.
 Following standard terminology, 
 if $\theta>2^{|\PP_I|}$ is regular, $M\prec H_\theta$ and $\PP_I\in M$, 
then a condition 
$B\in \PP_I$ is called \emph{$M$-generic} if for every maximal antichain $A\in M$ in $\PP_I$, $A\cap M$ is predense below $B$ in $\PP_I$. 
 $M$-generic conditions below $B\in \PP_I\cap M$ exist if and only if $A_{\PP_I,M}\cap B\notin I$ by \cite[Proposition 2.2.2]{MR2391923} and in this case $A_{\PP_I,M}\cap B$ is $M$-generic.  


From now on, we will assume that $\PP_I$ is a forcing on the reals by working with Borel codes instead of Borel sets. 
Thus $\PP_I$ is the set of Borel codes for Borel sets $A\notin I$. 
However, we still use the notation ``$x\in A$'' when $x$ is a real and $A\in \PP_I$ (thus $A$ is a Borel code) to mean that $x$ is an element of the set coded by $A$. 
Note that in this set-up, conditions $A\in \PP_I$ and nice $\PP_I$-names in $H_{\omega_1}$ will not be moved by transitive collapses. 

We shall use the following standard fact about proper forcings on the reals: 
for any $\PP_I$-name $\sigma$ for a real and any $p\in \PP_I$, there is a nice $\PP_I$-name $\tau\in H_{\omega_1}$ and some $q\leq p$ with $q\Vdash \sigma=\tau$ (see e.g. \cite[Proposition 2.11]{castiblanco2021preserving}). 




We now proceed with the proofs. 
The following level-by-level equivalence strengthens the equivalence of \ref{characterization of internal absoluteness via 1-step absoluteness and regularity 1} and \ref{characterization of internal absoluteness via 1-step absoluteness and regularity 2} in Theorem \ref{characterization of internal absoluteness via 1-step absoluteness and regularity}. 

\begin{lemma} 
\label{level-by-level equivalence of internal absoluteness and uniformization}
Let $n\geq1$. 
The following statements are equivalent for all proper forcings $\PP_I$: 
\begin{enumerate-(1)} 
\item 
\label{level-by-level equivalence of internal absoluteness and uniformization a}
Internal $\Sigma^1_n$ $\PP_I$-absoluteness. 

\item 
\label{level-by-level equivalence of internal absoluteness and uniformization b}
${\bf\Sigma}^1_n$ uniformization up to $I$. 

\item 
\label{level-by-level equivalence of internal absoluteness and uniformization c}
${\bf\Pi}^1_{n-1}$ uniformization up to $I$. 
\end{enumerate-(1)} 
\end{lemma} 

\begin{proof} 
Write $\PP=\PP_I$. 

 \ref{level-by-level equivalence of internal absoluteness and uniformization a} $\Rightarrow$ \ref{level-by-level equivalence of internal absoluteness and uniformization b}: 
Suppose that $R=\{(x,y)\mid \varphi(x,y)\}$ is a relation on $2^\omega$, where $\varphi(x,y)$ is a $\Sigma^1_n$-formula with an additional parameter that we take to be interpreted by a fixed real $x_0$, but it is suppressed to simplify the notation. 
In more detail, one could work with a $\Sigma^1_n$-formula $\chi(x,y,z)$ with  $R=\{(x,y)\mid \chi(x,y,x_0)\}$. 
As we mentioned above, we shall identify elements of $2^\omega$ with sets of natural numbers. 
Fix some $A\in \PP$ and take a countable $M\prec H_\theta$ as in \ref{level-by-level equivalence of internal absoluteness and uniformization a} that witnesses properness of $\PP$ with $A, x_0\in M$. 
Let $\bar{M}$ denote its transitive collapse and let $\bar{\PP}$ denote the image of $\PP$.  
In $\bar{M}$, let $\sigma$ be a name for the $\bar{\PP}$-generic real. 
We have two cases: 

\begin{case} 
There exists some $B\leq A$ with $B\Vdash \nexists y\ \varphi(\sigma,y)$. 
In this case, $C:=A_{\bar{\PP},\bar{M}}\cap B \subseteq \{x \mid \nexists y\ \varphi(x,y) \}$, so $C$ is disjoint from $p[R]$. 
Since $C\notin I$ by properness, $C$ is as required. 
\end{case} 

\begin{case} 
There exists some $B\leq A$ with $B\Vdash \exists y\ \varphi(\sigma,y)$. 
By fullness,\footnote{By fullness, we mean that if an existential statement is forced, then this is witnessed by a name.} there is a nice name $\tau\in \bar{M}$ with 
$B\Vdash \varphi(\sigma,\tau)$. 
We can assume that $\tau\in H_{\omega_1}$ by strengthening $B$ using the standard fact above (since $\PP$ is assumed to be a forcing on the reals). 
We further have $C:=A_{\bar{\PP},\bar{M}}\cap B\notin I$, since $\PP$ is proper. 
Now for any $x\in C$, $\bar{M}[x] \models \varphi(x,\tau^{(x)})$ and hence $H_\theta\models \varphi(x,\tau^{(x)})$ and $V \models \varphi(x,\tau^{(x)})$ by \ref{level-by-level equivalence of internal absoluteness and uniformization a}. 
Therefore the graph of the function $f\colon C\rightarrow 2^\omega$ defined by $f(x)=\tau^{(x)}$ is a subset of $R$. 
As noted above, $f$ is Borel measurable as required. 
\end{case} 

The implication \ref{level-by-level equivalence of internal absoluteness and uniformization b} $\Rightarrow$ \ref{level-by-level equivalence of internal absoluteness and uniformization c} is clear.

\ref{level-by-level equivalence of internal absoluteness and uniformization c} $\Rightarrow$ \ref{level-by-level equivalence of internal absoluteness and uniformization a}: 
Fix a large regular $\theta$ and a countable $M\prec H_\theta$ with $\PP\in M$.
Let $\bar{M}$ denote the transitive collapse of $M$ and let $\bar{\PP}$ be the image of $\PP$ under the collapse. 
Let $\tau\in H_{\omega_1}^{\bar{M}}$ be a nice name for a real. 
For \ref{level-by-level equivalence of internal absoluteness and uniformization a}, it suffices to show the implication 

$$ V\models \psi(\tau^g) \Longrightarrow \bar{M}[g] \models \psi(\tau^g)$$ 
for all $1\leq i\leq n$,  $\bar{\PP}$-generic filters $g\in V$ over $\bar{M}$ and formulas $\psi(x)=\exists y\ \varphi(x,y)$, where $\varphi$ is $\Pi^1_{i-1}$. 
We show this implication by induction on $i$. 
Note that the converse  follows from the inductive hypothesis for $\neg \varphi$. 
The case $i=1$ follows from $\Sigma^1_1$ absoluteness. 
Assume $i\geq 2$. 

Let $R:=\{ (x,y) \mid \varphi(\tau^{(x)},y)\}$. 
For a formula $\theta(x)$ and a nice name $\sigma$ for a real, write  
$$S_{\theta,\sigma} =\{x\in 2^\omega \mid \theta(\sigma^{(x)}) \}.$$ 
There are by \ref{level-by-level equivalence of internal absoluteness and uniformization c} densely many $A\in \bar{\PP}$ with the properties: 
\begin{enumerate-(a)} 
\item 
\label{uniformising sets 1} 
$A\subseteq S_{\psi,\tau}$ or $A\subseteq 2^\omega\setminus S_{\psi,\tau}$, and 
\item 
\label{uniformising sets 2} 
there is a Borel measurable function $f\colon A\rightarrow X$ whose graph is a subset of $R$.  
\end{enumerate-(a)} 
Now suppose that $x$ is $\bar{\PP}$-generic over $\bar{M}$ and $V \models \psi(\tau^x)$. 
Pick some $A$ with \ref{uniformising sets 1} and \ref{uniformising sets 2} such that $x\in A$. 
Since  $V \models\psi(\tau^{(x)})$, $x\in S_{\psi,\tau}$ and hence $A\subseteq S_{\psi,\tau}$ by  \ref{uniformising sets 1}. 
Pick a Borel measurable function $f\colon A \rightarrow X$ that uniformizes $R$ by \ref{uniformising sets 2}. 
Since $x\in A$, we have $V\models \varphi(\tau^x,f(x))$. 
By the inductive hypothesis, $\bar{M}[x] \models \varphi(\tau^x,f(x))$. 
Hence $\bar{M}[x] \models \psi(\tau^x)$ as required. 
\end{proof}

Now we are ready to prove Theorem \ref{characterization of internal absoluteness via 1-step absoluteness and regularity}. 
We will show the stronger level-by-level version: 

\begin{theorem} 
\label{main level-by-level equivalence} 
Let $n\geq1$. 
The following statements are equivalent for all proper forcings $\PP_I$: 
\begin{enumerate-(1)}
\item 
\label{main level-by-level equivalence 1}
 Internal $\Sigma^1_n$ $\PP_I$-absoluteness. 
 
\item 
\label{main level-by-level equivalence 3}
${\bf \Sigma}^1_n$ uniformization up to sets in $I$.

\item 
\label{main level-by-level equivalence 2}
${\bf \Pi}^1_{n-1}$ uniformization up to sets in $I$.

\item 
\label{main level-by-level equivalence 4}
$1$-step $\Sigma^1_{n+1}$ $\PP_I$-absoluteness and $I$-measurability of all ${\bf\Sigma}^1_n$ sets. 
\end{enumerate-(1)} 
\end{theorem} 

\begin{proof}
Write $\PP=\PP_I$. 
The equivalence of \ref{main level-by-level equivalence 1}, \ref{main level-by-level equivalence 3} and \ref{main level-by-level equivalence 2} was shown in Lemma \ref{level-by-level equivalence of internal absoluteness and uniformization}. 

\ref{main level-by-level equivalence 1} $\Rightarrow$ \ref{main level-by-level equivalence 4}: 
Take $\theta$ and $M\prec H_\theta$ as in the definition of internal $\Sigma^1_n$ $\PP$-absoluteness. 
Let $\bar{M}$ be the transitive collapse of $M$ and let $\bar{\PP}$ be the image of $\PP$ under the collapse. 
For $1$-step $\Sigma^1_{n+1}$ $\PP$-absoluteness, 
it suffices to have $\bar{M}\prec_{\Sigma^1_{n+1}}\bar{M}[g]$ for any $\bar{\PP}$-generic filter $g\in V$ over $\bar{M}$. 
To see that this suffices, note that if $p$ forces a $\Sigma^1_{n+1}$-statement over $V$, it also forces this over models $\bar M$ as above. 
Let $g$ be $\bar \PP$-generic over $\bar M$ with $p\in g$. 
The statement then holds in $\bar M$ and thus in $V$. 
Returning to the proof of  $\bar{M}\prec_{\Sigma^1_{n+1}}\bar{M}[g]$, 
take any $\Pi^1_n$-formula $\varphi(x,y)$ with additional parameters in $\bar M$, let $\psi(x)$ denote the formula $\exists y\ \varphi(x,y)$ and suppose that $\bar M[g] \models \psi(x)$, where $x\in \bar M$ is a real. 
Then $V\models \psi(x)$ holds by internal $\Sigma^1_n$ $\PP$-absoluteness applied to $M$, $\varphi$ and a witness $y$ for $\varphi$ in $\bar M[g]$. 
Since $M\prec H_\theta$, we have $M\models \psi(x)$ as required.

It remains to show that any ${\bf\Sigma}^1_n$ set $A$ is $I$-measurable. 
Recall that \ref{main level-by-level equivalence 1} $\Rightarrow$ \ref{main level-by-level equivalence 3} by Lemma \ref{level-by-level equivalence of internal absoluteness and uniformization}. 
By \ref{main level-by-level equivalence 3} applied to the relation $A \times\{0\}$, we have two cases. 
In the first case, there exists an $I$-positive set $B$ disjoint from $A$. In the second case, there exists an $I$-positive set $B$ and a Borel measurable subfunction $f\colon B\rightarrow 2^\omega$ of $A\times \{0\}$. 
Then $B\subseteq A$.

\ref{main level-by-level equivalence 4} $\Rightarrow$ \ref{main level-by-level equivalence 1}: 
Take some $M\prec H_\theta$ witnessing properness and let $\bar{M}$ denote the transitive collapse of $M$. 
Let $\bar{\PP}$ denote the image of $\PP$ under the transitive collapse. 
Let $\sigma\in M$ be a nice $\PP$-name for a real and $\bar{\sigma}$ its image under the collapse.  
Take any $\Sigma^1_n$ formula $\varphi(x)$ with additional parameters in $M$. 
Suppose that $x$ is a $\bar{\PP}$-generic real over $\bar{M}$. 
For \ref{main level-by-level equivalence 1}, it suffices to show: 
$$\bar{M}[x]\models \varphi(\bar{\sigma}^x) \Longleftrightarrow V\models \varphi(\bar{\sigma}^x).$$

We first prove two general claims. 
Suppose that $\tau\in H_{\omega_1}$ is a nice name for a real. 
Recall that $S_{\varphi,\tau} =\{x\in 2^\omega \mid \varphi(\tau^{(x)}) \}$ from the proof of Lemma \ref{level-by-level equivalence of internal absoluteness and uniformization}.
Call a set $A\in\PP$ \emph{$\tau$-decisive} if 
$\tau\in H_{\omega_1}$ is a nice name for a real, $A\Vdash \sigma=\tau$ and either $A\subseteq S_{\varphi,\tau}$ or $A\subseteq 2^\omega\setminus S_{\varphi,\tau}$. Call $A$ \emph{decisive} if it is $\tau$-decisive for some $\tau$. 

\setcounter{cl}{0}
\begin{claim} 
\label{set of decisive conditions is dense} 
The set of decisive $A\in \PP$ is dense. 
\end{claim} 

\begin{proof} 
Since $\PP$ is a proper forcing on the reals, the set of $A\in\PP$ such that $A\Vdash \sigma=\tau$ holds for some nice name $\tau\in H_{\omega_1}$ is dense. 
For any such $A$, the set $S_{\varphi,\tau}=\{x\in \omega^\omega\mid \varphi(\tau^{(x)})\}$ is ${\bf \Sigma}^1_n$ and hence $I$-measurable by \ref{main level-by-level equivalence 4}. 
Therefore, there is some $B\leq A$ with either $B\subseteq S_{\varphi,\tau}$ or $B\subseteq\omega^\omega\setminus S_{\varphi,\tau}$ as required. 
\end{proof}

\begin{claim} 
\label{properties of decisive sets} 
If $A\in M$ is $\tau$-decisive for some $\tau\in M$, then: 
\begin{enumerate-(a)} 
\item 
\label{properties of decisive sets 1}
If $A\subseteq S_{\varphi,\tau}$, then $A\Vdash \varphi(\tau)$. 
\item 
\label{properties of decisive sets 2}
If $A\subseteq \omega^\omega \setminus S_{\varphi,\tau}$, then $A\Vdash \neg \varphi(\tau)$.
\end{enumerate-(a)} 
\end{claim} 

\begin{proof} 
Recall that the elements of $\PP$ are Borel codes. 
In particular, conditions $A\in \PP$ and nice $\PP$-names in $H_{\omega_1}$ are not moved in the transitive collapse of $M$. 

\ref{properties of decisive sets 1}: 
Towards a contradiction, suppose that $A\not\Vdash \varphi(\tau)$. 
Since $\bar{M}\cong M\prec H_\theta$, we have $A\not\Vdash^{\bar{M}} \varphi(\tau)$. 
Hence there is some $B\leq A$ in $\bar{M}$ with $B\Vdash^{\bar{M}}\neg\varphi(\tau)$.
Then $\exists x\in B\ (\neg\varphi(\tau^{(x)}))$ holds in any $\bar{\PP}$-generic extension of $\bar{M}$ by a filter containing $B$. 
By $1$-step $\Sigma^1_{n+1}$ $\PP$-absoluteness by \ref{main level-by-level equivalence 4}, $\bar{M} \models \exists x\in B\ (\neg\varphi(\tau^{(x)}))$. 
Since $\bar{M}\cong M\prec H_\theta$, we have $V \models \exists x\in B\ (\neg\varphi(\tau^{(x)}))$. 
But this contradicts $B\subseteq A \subseteq S_{\varphi,\tau}$ (here we identify $A$ and $B$ with the Borel sets which they code).

\ref{properties of decisive sets 2}: 
An argument analogous to \ref{properties of decisive sets 1} works. 
The roles of $\varphi$ and $\neg\varphi$ are switched and $1$-step $\Sigma^1_n$ $\PP$-absoluteness by \ref{main level-by-level equivalence 4} is used. 
\end{proof} 

We can now prove the theorem with the help of the previous claims. 
Recall that $x$ is $\bar{\PP}$-generic over $\bar{M}$. 
By Claim \ref{set of decisive conditions is dense} applied in $\bar{M}$, there exists some decisive $A\in M$ with $x\in A$. 
Since $\bar{M}\cong M\prec H_\theta$, there is some $\tau\in M$ such that $A$ is $\tau$-decisive. 
We first claim that $\bar{\sigma}^x=\tau^x=\tau^{(x)}$. 
To see this, note that $A\Vdash \sigma=\tau$, since $A$ is $\tau$-decisive. 
Since $\bar{M} \cong M \prec H_\theta$, we then have $A\Vdash^{\bar{M}} \bar{\sigma}=\tau$. 
Since $x\in A$, $A$ is in the filter induced by $x$. 
Hence $\bar{\sigma}^x=\tau^x$. 
We further have $\tau^x=\tau^{(x)}$, since $x$ is $\bar{\PP}$-generic over $\bar{M}$. 

It thus suffices to show: 
$$\bar{M}[x]\models \varphi(\tau^x) \Longleftrightarrow V\models \varphi(\tau^{(x)}).$$ 
To see this, consider the two cases in Claim \ref{properties of decisive sets}. 
First suppose that $A\subseteq S_{\varphi,\tau}$. 
We have $\bar{M}[x]\models \varphi(\tau^{x})$, since $A\Vdash \varphi(\tau)$ holds by Claim \ref{properties of decisive sets} and therefore $A\Vdash^{\bar{M}} \varphi(\tau)$ using $\bar{M}\cong M\prec H_\theta$. 
Moreover, $V\models \varphi(\tau^{(x)})$ since $x\in A\subseteq S_{\varphi,\tau}$. 
Finally, suppose that $A\subseteq\omega^\omega\setminus S_{\varphi,\tau}$. 
Similar to the previous case, we have $\bar{M}[x]\models \neg\varphi(\tau^{x})$ and $V\models \neg\varphi(\tau^{(x)})$. 
\end{proof}

\section{Consistency strength and examples} 

We round off this paper by some remarks of the consistency strength of the statements shown to be equivalent in Theorem \ref{characterization of internal absoluteness via 1-step absoluteness and regularity}. More precisely, we consider the strength of internal projective absoluteness, first for the class of all forcings and then for the specific examples of Cohen forcing and random forcing.

\subsection{The strength of internal projective absoluteness} 

  Woodin proved that 2-step projective absoluteness (for the class of all forcings) holds in generic extensions collapsing certain large cardinals. More precisely, he showed that if $\kappa_1 < \dots < \kappa_n$ are strong cardinals, then 2-step $\SIGMA^1_{n+3}$ generic absoluteness holds in $\Col(\omega, 2^{2^{\kappa_n}})$-generic extensions, see \cite[Corollary 4.7]{St09}. Wilson \cite{Wi} improved Woodin's result by showing that 2-step $\SIGMA^1_{n+3}$ generic absoluteness already holds in $\Col(\omega, 2^{\kappa_n})$-generic extensions of the universe. Woodin's argument shows in fact that internal projective absoluteness can be forced from infinitely many strong cardinals. More precisely, if
  $\lambda$ is a limit of strong cardinals, then internal projective absoluteness holds in $V[G]$
  where $G$ is a $\Col(\omega,\lambda)$-generic filter over $V$. In addition, local versions of this result hold analogous to \cite[Corollary 4.7]{St09} and \cite[Theorem 1.1]{Wi} for 2-step projective generic absoluteness.

  Recall that for a class $\CF$ of forcings given by a definition without parameters, $2$-step absoluteness states that for any forcing $\PP$ in $\CF$, any $\PP$-name $\dot{\QQ}$ for a forcing in $\CF$ and any $\PP*\dot{\QQ}$-generic filter $G*H$ over $V$, $V[G]$ and $V[G*H]$, 
  \[ V[G] \models \varphi(x) \Longleftrightarrow V[G*H] \models \varphi(x) \] 
  holds for every projective formula $\varphi(v_0)$ and every real $x \in V[G]$. 
  We further define  \emph{relative projective absoluteness} to hold for $\CF$ if for all generic extensions $V[G]$ and $V[H]$ for forcings in $\CF$ with $V[G]\subseteq V[H]$, $V[G]$ and $V[H]$,  \[ V[G] \models \varphi(x) \Longleftrightarrow V[H] \models \varphi(x) \] 
  holds for every projective formula $\varphi(v_0)$ and every real $x \in V[G]$.
  
  These two notions are clearly equivalent if $\CF$ is closed under $2$-step iterations and quotients, i.e., for any two generic extensions $V[G]$ and $V[H]$ of $V$ by forcings in $\CF$ with $V[G]\subseteq V[H]$, $V[H]$ is itself a generic extension of $V[G]$ by a forcing in $\CF$. 
  This is the case for Cohen forcing and random forcing (see \cite[Lemma 3.2.8]{BaJu95} and \cite{kanovei2018intermediate}). 
  
  It is easy to see that internal projective absoluteness implies relative projective absoluteness. 

\begin{lemma}\label{lem:1and2stepabs}
  For any class $\CF$ of forcing notions, internal projective $\CF$-absoluteness implies relative projective absoluteness for forcings in $\CF$. 
\end{lemma}
\begin{proof} 
Take any countable $\bar M \cong M \prec H_\theta$ witnessing internal projective absoluteness. 
It suffices to show that relative projective absoluteness holds for generic extensions of $\bar M$ via forcing notions in the transitive collapse of $\CF \cap M$. 
To this end, suppose that $\PP,\QQ \in \bar M$ are such forcing notions, $g$ is $\PP$-generic over $\bar M$ and  $h$ is $\PP$-generic over $\bar M$ with $\bar M[g]\subseteq \bar M[h]$. 
Let $\varphi(v_0)$ be a projective formula and $x \in \bar M[g]$ some real. Then \[ \bar M[g] \models \varphi(x) \Longleftrightarrow H_\theta \models \varphi(x) \Longleftrightarrow \bar M[h] \models \varphi(x), \] by internal projective absoluteness applied to $\bar M[g]$ and $\bar M[h]$.
\end{proof} 

Lemma \ref{lem:1and2stepabs} together with a result by Hauser, see
\cite{Ha95}, yields the following fact.

\begin{fact}
  Internal projective absoluteness for the class of all forcings implies the existence of an inner model with infinitely many strong cardinals.
\end{fact}

Therefore, the consistency strength of internal projective absoluteness is exactly $\omega$ strong cardinals.


\subsection{Cohen forcing} 

In this section we look more closely at internal projective absoluteness for Cohen forcing. Is is not hard to see that internal projective $\CC$-absoluteness, where $\CC$ denotes Cohen forcing, does not follow from projective Cohen $2$-step absoluteness.

\begin{lemma}
\label{Cohen 2-step absoluteness does not imply internal Cohen absoluteness} 
Internal projective $\CC$-absoluteness does not follow from projective 2-step absoluteness for $\CC$.
\end{lemma}
\begin{proof}
We will use the following well-known claim. We include a proof for the reader's convenience. 

\setcounter{cl}{0}
\begin{claim} \label{cl:projective non-full set is meager} 
Suppose that $M$ is an inner model of $V$ and $R={}^\omega 2\cap M\neq 2^\omega$. If $R$ has the property of Baire, then it is meager. 
\end{claim} 
\begin{proof}
First fix some notation. 
Write $(x+y)(i)=x(i)+y(i)$ mod $2$ for $x,y\in {}^\omega 2$ and $x+Y=\{x+y\mid y\in Y\}$ for $x\in {}^\omega 2$ and $Y\subseteq {}^\omega 2$. 
Let $R$ and $M$ be as in the statement of the claim and suppose that $R$ is not meager. Since $R$ has the property of Baire, there is some $t\in 2^{<\omega}$ such that $R\cap N_t$ is comeager in $N_t$, where $N_t = \{ s \in {}^\omega 2 \mid s \text{ extends } t \}$. 
Let $x\in N_{0^{|t|}}\setminus R$, in particular $x \neq 0$. Then $R\cap N_t$ and $x+(R\cap N_t)$ are disjoint as if $y \in (R\cap N_t) \cap (x+(R\cap N_t))$ then there is some real $z \in R\cap N_t$ such that $y = x+z$. But then $x = y+z \in R$ as $R={}^\omega 2\cap M$ and hence closed under addition. This contradicts our choice of $x \notin R$. Note that $x+(R\cap N_t)$ is comeager in $N_t$ as $R\cap N_t$ is comeager in $N_t$, ``$+ x$'' is a homeomorphism and $0^{|t|}$ is a subsequence of $x$. But this contradicts the fact that $R\cap N_t$ and $x+(R\cap N_t)$ are disjoint subsets of $N_t$. 
\end{proof} 

 It is well-known that the set of ground model reals is not meager in $\Add(\omega,\omega_1)$-generic extensions, see \cite{Ku84}. Therefore, using the claim, the set of ground model reals in $\Add(\omega,\omega_1)$-generic extensions cannot have the property of Baire. Recall that by Theorem \ref{characterization of internal absoluteness via 1-step absoluteness and regularity} and the remark before the statement of Theorem \ref{characterization of internal absoluteness via 1-step absoluteness and regularity} in the introduction, internal projective $\CC$-absoluteness implies that all projective sets have the property of Baire. See also \cite[Lemma 5.7]{MSSW} for a direct proof of this fact. So no non-trivial $\Add(\omega,\omega_1)$-generic extension of $L$ satisfies internal projective $\CC$-absoluteness as the set of ground model reals is projective if the ground model is $L$. In addition, it is easy to see that projective Cohen $2$-step absoluteness holds in these generic extensions.
\end{proof}

Using the claim in the previous proof, 
we can obtain some additional consequences of internal projective $\CC$-absoluteness. It implies, for example, that for any real $x$, $L[x]$ is meager and there is a Cohen real over $L[x]$. Moreover, it implies by \cite[Lemma 4]{Wo82} that if $\omega_1=\omega_1^L$, there is a real $x$ such that there is no random real over $L[x]$. 

Nevertheless, internal projective $\CC$ can be forced over a $\ZFC$ model. For example, it holds by Theorem \ref{characterization of internal absoluteness via 1-step absoluteness and regularity} in Shelah's model for uniformization up to a meager set \cite{Sh84, JR93}.

\subsection{Random forcing}

In this section, we discuss the consistency strength of internal projective $\RR$-absoluteness for random forcing $\RR$. Recall that 2-step absoluteness for random forcing can be forced over $\ZFC$; it holds, for example, in a generic extension by the random algebra for uncountably many generators. But by Theorem \ref{characterization of internal absoluteness via 1-step absoluteness and regularity}, internal projective $\RR$-absoluteness implies that all projective sets of reals are Lebesgue measurable. In particular, an inaccessible cardinal is a consistency strength lower bound for internal projective $\RR$-absoluteness. Theorem \ref{characterization of internal absoluteness via 1-step absoluteness and regularity} combined with results of Solovay show that this lower bound is optimal: Internal projective $\RR$-absoluteness is equivalent to projective uniformization up to a null set. Solovay showed that in his famous model after collapsing an inaccessible cardinal projective uniformization up to a null set holds \cite{So69}. We summarize this discussion in the following lemma.

\begin{lemma}
The consistency strength of internal projective $\RR$-absoluteness is an inaccessible cardinal. In particular, internal projective $\RR$-absoluteness does not follow from projective 2-step absoluteness for $\RR$.
\end{lemma}

\section{Open questions} 

We close this paper with two questions about projective uniformization up to small sets. In Solovay's model projective uniformization holds up to meager sets and up to null sets. This raises the next question. 

\begin{question} 
In Solovay's model, does projective uniformization up to $I$ hold for every $\sigma$-ideal $I$ such that $\PP_I$ is proper? 
\end{question}

As uniformization up to $I$ implies $I$-measurability, this would imply a positive answer to the question whether all sets are $I$-measurable in Solovay's model. 
This was asked by Khomskii, Ikegami and others (see \cite[Question 6.3]{Ikegami2021}).\footnote{An answer is claimed in \cite[Proposition 2.2.8]{Kh12}, but Ikegami found a gap \cite[Section 6]{Ikegami2021}.} 


We have seen in Lemma \ref{Cohen 2-step absoluteness does not imply internal Cohen absoluteness} that internal projective absoluteness for Cohen forcing does not follow from 2-step projective Cohen absoluteness. In terms of regularity properties, it is natural to conjecture that projective uniformization up to a meager set does not follow from the fact that all projective sets have the property of Baire. Surprisingly, this is open. 

\begin{question}
Does the Baire property for all projective sets of reals imply projective uniformization up to meager sets?
\end{question}

Shelah has shown that both the Baire property for all projective sets and projective uniformization up to meager set are consistent relative to $\ZFC$ \cite{Sh84, JR93}. He produced two different models for this and it is not clear whether, for example, projective uniformization up to a meager set fails in his first model for the Baire property.




\bibliographystyle{plainnat}
\bibliography{ref} 

\end{document}